\documentclass[a4paper]{amsart}
\usepackage[english]{babel}
\usepackage{amsmath,amssymb,amsthm}

\usepackage{graphicx}

\usepackage[pdftex]{color}
\usepackage[bookmarks=true,hyperindex,pdftex,colorlinks,citecolor=blue]{hyperref}

\theoremstyle{plain}
\newtheorem{theorem}{Theorem}[section]
\newtheorem{corollary}[theorem]{Corollary}
\newtheorem{prop}[theorem]{Proposition}
\newtheorem{proposition}[theorem]{Proposition}

\theoremstyle{definition}

\newtheorem{example}[theorem]{Example}

\newtheorem{definition}[theorem]{Definition}

 \DeclareMathOperator{\re}{Re\,}

 \DeclareMathOperator{\dist}{dist\,}

\newcommand{\D}{\mathbb{D}}
\newcommand{\K}{\mathbb{K}}

\newcommand{\R}{\mathbb{R}}
\newcommand{\N}{\mathbb{N}}

\newcommand{\inner}[1]{\ensuremath{\left\langle #1\right\rangle}}

\newcommand{\norm}[1]{\ensuremath{\lVert#1\rVert}}

\newcommand{\nor}[1]{\ensuremath{\left\|#1\right\|}}

\newcommand{\eps}{\varepsilon}

\newcommand{\caconv}{\overline{\mathrm{aconv}}}

\renewcommand{\leq}{\leqslant}
\renewcommand{\geq}{\geqslant}

\begin{document}

\dedicatory{Dedicated to Richard M. Aron on the occasion of his 70th birthday}

\title[On Banach spaces with the Approximate hyperplane series property]{On Banach spaces with the \\ approximate hyperplane series property}

\author[Choi]{Yun Sung Choi}
\address[Choi]{Department of Mathematics, POSTECH, Pohang (790-784), Republic of Korea}
\email{\texttt{mathchoi@postech.ac.kr}}

\author[Kim]{Sun Kwang Kim}
\address[Kim]{Department of Mathematics, Kyonggi University , Suwon 443-760, Republic of Korea}
\email{\texttt{sunkwang@kgu.ac.kr}}

\author[Lee]{Han Ju Lee}
\address[Lee]{Department of Mathematics Education,
Dongguk University - Seoul, 100-715 Seoul, Republic of Korea}
\email{\texttt{hanjulee@dongguk.edu}}

\author[Mart\'{\i}n]{Miguel Mart\'{\i}n}
\address[Mart\'{\i}n]{Departamento de An\'{a}lisis Matem\'{a}tico,
Facultad de Ciencias,
Universidad de Granada,
E-18071 Granada, Spain}
\email{\texttt{mmartins@ugr.es}}


\thanks{The first author was supported by by Basic Science Research Program
through the National Research Foundation of Korea (NRF) funded by the
Ministry of Education (No.2010-0008543 and No. 2013053914).
The third author partially supported by Basic Science Research Program through the National Research Foundation of Korea(NRF) funded by the Ministry of Education, Science and Technology (NRF-2012R1A1A1006869).
The fourth author partially supported by Spanish MICINN and FEDER project no.\ MTM2012-31755 and Junta de Andaluc\'{\i}a and FEDER grants FQM-185 and P09-FQM-4911.}

\begin{abstract}
We present a sufficient condition for a Banach space to have the approximate hyperplane series property (AHSP) which actually covers all known examples. We use this property to get a stability result to vector-valued spaces of integrable functions. On the other hand, the study of a possible Bishop-Phelps-Bollob\'{a}s version of a classical result of V.~Zizler leads to a new characterization of the AHSP for dual spaces in terms of $w^*$-continuous operators and other related results.
\end{abstract}

\date{July 25th, 2014}

\maketitle

\section{Introduction}

Given two (real or complex) Banach spaces $X$ and $Y$, we write $L(X,Y)$ for the Banach space of all bounded linear operators from $X$ into $Y$, endowed with the operator norm. We use the symbol $B_X$ and $S_X$ to denote, respectively, the closed unit ball and the unit sphere of $X$. The topological dual of $X$ is denote by $X^*$. An operator $T\in L(X, Y)$ is said to be \emph{norm-attaining} if there is $x\in S_X$ such that $\|Tx\|= \|T\|$. We write $NA(X,Y)$ for the subset of $L(X,Y)$ consisting on all norm-attaining operators. After the celebrated Bishop-Phelps theorem, which says that the set of all norm-attaining linear functionals is dense in $X^*$, it was a natural question to study whether the set of norm-attaining linear operators is dense in $L(X,Y)$ for all Banach spaces $X$ and $Y$. In 1963, Lindenstrauss \cite{Lindens} gave a negative answer to this question and showed, among many other things, that $NA(\ell_1, Y)$ is dense in $L(\ell_1,Y)$ for all Banach spaces $Y$.

Bollob\'as sharpened the Bishop-Phelps theorem to which is now called the Bishop-Phelps-Bollob\'as theorem: given a Banach space $X$, if $x^*\in S_{X^*}$ and $x\in S_X$ satisfy $|x^*(x) - 1| \leq \eps^2 /2$ for some $0<\eps<1/2$, then there exist $y\in S_X$ and $y^*\in S_{X^*}$ such that $y^*(y)=1$, $\|y-x\|<\eps + \eps^2$ and $\|y^*-x^*\|\leq \eps$. See \cite{C-K-M-M-R} for a sharper version of this result. In 2008, Acosta, Aron,  Garc\'ia and Maestre \cite{AAGM2} introduced the notion of Bishop-Phelps-Bollob\'as theorem for operators. Precisely, a pair $(X, Y)$ of Banach spaces is said to have the \emph{Bishop-Phleps-Bollob\'as property} (\emph{BPBp} in short) if given $\eps>0$,  there is $\eta(\eps)>0$ such that whenever $T\in L(X,Y)$ with $\|T\|=1$ and $x\in S_X$ satisfy $\|Tx\|>1-\eta(\eps)$, there exist $S\in L(X, Y)$ and $y\in S_X$ such that
\[
\|S- T\|<\eps, \ \ \ \|x-y\|<\eps, \ \ \ \text{and} \ \ \ \ \|S\|=\|Sy\|=1.
\]
In the same paper \cite{AAGM2}, the authors characterize Banach spaces $Y$ for which the pair ($\ell_1, Y$) has the BPBp in terms of convex series and use the result to show that there is a reflexive Banach space $Y$ such that $(\ell_1, Y)$ does not have the BPBp (while, as commented above, $NA(\ell_1,Y)$ is dense in $L(\ell_1,Y)$ for every $Y$). Specifically,
$(\ell_1,Y)$ has the BPBp if and only if $Y$ has the \emph{approximate hyperplane
series property} (\emph{AHSP}): for every $\eps>0$ there exists $0<\eta(\eps)<\eps$ such that for every
sequence $(y_k) \subset S_Y$ and for every convex series $\sum_{k=1}^{\infty}\alpha_k$ with
$$
\left\|\sum_{k=1}^{\infty}\alpha_k y_k\right\| > 1-\eta(\eps),
$$
there exist $A\subset \mathbb{N}$, $y^*\in S_{Y^*}$ and $\{z_k\, :\, k\in A\}\subset S_Y$ satisfying
\begin{itemize}
\item[(1)] $\sum\limits_{k\in A} \alpha_k>1-\eps$,
\item[(2)] $\|z_k-y_k\|<\eps$ and $y^*(z_k)=1$ for all $k\in A$.
\end{itemize}
Let us remark that the definition of the AHSP does not change if we replace infinite sequences by finite (but of arbitrary length) sequences.

The following spaces are known to have the AHSP
\cite{AAGM2,ChoiKimSK}: finite dimensional spaces, uniformly convex
spaces, spaces with property $\beta$ (see definition in section \ref{sec:AHP}), and the so-called lush spaces (see definition in section \ref{sec:AHP}). In particular, $C_0(L)$ spaces, $L_1(\mu)$ spaces, the disc algebra $A(\D)$ and $H^{\infty}(\D)$, all have the AHSP.

In section~\ref{sec:AHP} of the present paper, we introduce a sufficient condition for the AHSP, called the AHP (see Definition~\ref{def:AHP}), which actually subsumes all previously known examples. We use the AHP to present new examples of spaces with the AHSP, namely, the spaces $L_1(\mu,X)$ when $\mu$ is an arbitrary measure and $X$ is finite-dimensional, uniformly convex, lush or has property $\beta$.

Finally, we start section~\ref{sec:Zizler} by showing that there is no BPB version of the theorem of Zizler \cite{Zizler} which states that the set of operators between two Banach spaces whose
adjoints attain their norm is dense in the space of all operators. Then, we exploit this idea getting several results: a new characterization of the AHSP for dual spaces in terms of $w^*$-continuous operators, a characterization of pairs $(X,X^*)$ having the AHSP (a concept recently introduced in \cite{ABGM}, see its definition in section~\ref{sec:Zizler}), and a proof that $(C_0(K,Y),C_0(K,Y)^*)$ has the AHSP for every locally compact
Hausdorff space $K$ and every uniformly smooth space $Y$, thereby extending a result in \cite{ABGM} to the vector-valued case.

\vspace*{2em}

\textbf{Acknowledgement:} This research was initiated during the visit of the first, third and fourth authors to the Department of Mathematical Sciences, Kent State University in 2012. They would like to thank the department and specially Richard Aron for their hospitality and fruitful discussions about this research topic.


\section{A sufficient condition for the AHSP}\label{sec:AHP}

We devote this section to study a sufficient condition for the AHSP, which actually covers all known examples, and which will be useful to provide new examples of the form $L_1(\mu,X)$. We need some notation. Let $X$ be a Banach space. A \emph{face} of $B_X$ is a non-empty subset of the form
$$
F(x^*) := \bigl\{x\in S_X\, :\, \re x^*(x)=1\bigr\}
$$
for a suitable $x^*\in S_{X^*}$ attaining its norm. A subset $C\subseteq S_{X^*}$ is said to be \emph{norming} if
$$
\|x\|=\sup\{|x^*(x)|\,:\,x^*\in C\} \quad (x\in X)
$$
and it is said to be \emph{rounded} if $\theta C=C$ for every $\theta\in \K$ with $|\theta|=1$.

\begin{definition}\label{def:AHP}
A Banach space $X$ is said to have the {\it approximate hyperplane property} (\emph{AHP} in short) if there exist a function $\delta:(0,1) \longrightarrow (0, 1)$ and a norming subset $C$ of $S_{X^*}$ for $X$ satisfying the following:\\
Given $\eps>0$, there is a function $\Upsilon_{X,\eps}: C\longrightarrow S_{X^*}$ such that if $x^*\in C$ and $x\in S_X$ satisfy $\re x^*(x)>1-\delta(\varepsilon)$, then $\dist\bigl(x,F(\Upsilon_{X,\eps}(x^*))\bigr)<\eps$.

Observe that by a routine argument, we may suppose the set $C$ to be rounded. It is also straightforward to prove that it is enough to check the property just for a dense subset of $S_X$.
\end{definition}

As announced in the introduction, we show that the AHP implies the AHSP. Actually, more can be said.

\begin{prop}\label{prop-AHPimpliesAHSP}
Let $X$ be a Banach space. Suppose that there is a function $\delta:(0,1)\longrightarrow (0,1)$ such that for every finite-dimensional subspace $Y$ of $X$, there exists a subspace $Z$ of $X$ containing $Y$ and having the AHP with the function $\delta$. Then, $X$ has the AHSP. In particular, the AHP implies the AHSP.
\end{prop}

\begin{proof}
Let $0<\eps<1$ and write $\delta_1(\eps) = \min\{ \delta(\eps), \eps \}$.
Consider a finite convex combination $\sum_{j=1}^n \alpha_j x_j $ of elements of $S_X$ satisfying $\norm{\sum_{j=1}^n \alpha_j x_j } > 1-\delta_1^2$. By hypothesis, there is a subspace $Z$ of $X$ with the AHP containing the subspace spanned by $\{x_j\,:\,j=1,\ldots,n\}$; we denote by $C$ the norming subset of $S_{Z^*}$ given by the AHP of $Z$, which we may and do suppose that it is rounded. Then, there exists $x^*\in C$ such that
\[
\re x^*\Bigl(\sum_{j=1}^n \alpha_j x_j \Bigr)=\re \sum_{j=1}^n \alpha_j x^*(x_j) > 1-\delta_1(\eps)^2.
\]
Setting
$$
A= \bigl\{ j \,:\, 1\leq j \leq n,\ \re x^*(x_j) > 1-\delta_1(\eps)\bigr\},
$$
we can get easily that
$$
\sum_{j\in A} \alpha_j >1-\delta_1(\eps)\geq 1-\eps
$$
(see the proof of \cite[Proposition~2.1]{LeeMartin}). By the AHP of $Z$, $\dist\bigl(x_j,F(\Upsilon_{Z,\eps}(x^*))\bigr)<\eps$ for all $j\in A$. Letting $y^*$ be a Hahn-Banach extension of $\Upsilon_{Z,\eps}(x^*)$ to the whole of $X$, then we get $\dist\bigl(x_j,F(y^*)\bigr)<\eps$ for $j=1,\ldots,n$, as desired.
\end{proof}

Examples of spaces with AHP appeared already (without this name) in the literature. Actually, the usual way to provide examples of Banach spaces with the AHSP has been to prove that they have the AHP. Let us discuss the main examples.

We start with the easiest case: uniformly convex spaces.

\begin{prop}
Every uniformly convex Banach space has the AHP. Besides, here $C$ is the whole dual unit sphere and $\Upsilon$ is the identity.
\end{prop}

This result appeared in \cite[Lemma~13]{ChoiKimSK} (without this notation) and also follows easily from \cite[Lemma~2.1]{ABGM}.

Next, it is shown in \cite[Lemma~3.4]{AAGM2} that every finite-dimensional Banach space has the AHP.

\begin{prop}\label{prop-fin-dim-AHP}
Every finite-dimensional Banach space has the AHP. Besides, here $C$ is the whole dual unit sphere but $\Upsilon$ is not, in general, equal to the identity.
\end{prop}

It follows from the result above, Proposition~\ref{prop-AHPimpliesAHSP} and the results in \cite[\S 4]{ACKLM-1}, that property AHP is not stable by infinite $c_0$-, $\ell_1$- or $\ell_\infty$-sums.

\begin{example}
{\slshape AHP is not stable by infinite $c_0$-, $\ell_1$- or $\ell_\infty$-sums.\ } Indeed, a sequence $\{Y_k\}_{k\in\N}$ of finite-dimensional spaces is presented in \cite[\S 4]{ACKLM-1} such that its $c_0$-, $\ell_1$- and $\ell_\infty$-sums fail the AHSP. Now, all the $Y_k$'s have the AHP (Proposition~\ref{prop-fin-dim-AHP}) and the sums fail it by Proposition~\ref{prop-AHPimpliesAHSP}.
\end{example}

The next family of examples we present here is the one of lush spaces. A Banach space $X$ is said to be \emph{lush} \cite{BKMW} if for every $x_0,y_0\in S_X$ and every $\eps>0$, there is a slice $S:=\{x\in B_X\,:\, \re x^*(x)>1 -\eps\}$ with $x^* \in S_{X^*}$ such that $x_0 \in S$ and the distance from $y_0$ to the absolutely convex hull of $S$ is smaller than $\eps$. We refer to \cite{BKMM,BKMW,KMMP2009,LeeMartin} for information about lush spaces. Among lush spaces we may find $C(K)$ spaces, $L_1(\mu)$ spaces and their isometric preduals, the disk algebra, $H^\infty(\D)$, and finite-codimensional subspaces of $C[0,1]$.

\begin{prop}\label{prop-lush-separable}
Every separable lush space has the AHP. Besides, the function $\delta$ does not depend on the particular lush space, $C$ is not the whole dual unit sphere and $\Upsilon$ is the identity.
\end{prop}

A proof of this result is contained in the proof of \cite[Proposition~2.1.c]{LeeMartin}. We are going to comment this proof here for the sake of completeness and also in order to extend the result to some non-separable spaces. Indeed, let $X$ be a separable lush space. By \cite[Theorem~4.3]{KMMP2009} and \cite[Corollary~3.5]{AcoBecRod}, there exists a rounded subset $C$ of $S_{X^*}$ norming for $X$ such that
\begin{equation}\label{eq:lush-CL}
B_X=\caconv\bigl(F(x^*)\bigr)
\end{equation}
for every $x^*\in C$, where $\caconv(\cdot)$ denotes the absolutely closed convex hull. With this in mind, one may follows the proof of \cite[Proposition~2.1.c]{LeeMartin} to get Proposition~\ref{prop-lush-separable} with a function $\delta$ (which is independent of $X$), the norming set $C$ and $\Upsilon$ equals to the identity.

Therefore, the key ingredient of the proof is to get the existence of a norming set $C$ such that \eqref{eq:lush-CL} holds for every element of $C$. Another family of spaces for which this happens is the one of almost-CL-spaces. A Banach space $X$ is said to be an \emph{almost-CL-space} if $B_X$ is the closed absolutely convex hull of every maximal convex subset of $S_X$. We refer the reader to \cite{MartPaya-CL} and references therein for more information and background. Almost-CL-spaces are lush, but the converse is not true \cite{BKMW}. The main examples of almost-CL-spaces are $C(K)$-spaces and $L_1(\mu)$-spaces. It is easy to show (see \cite[\S 2]{MartPaya-CL}) that if $X$ is an almost-CL-space, then there is a rounded subset $C$ of $S_{X^*}$ norming for $X$ such that \eqref{eq:lush-CL} holds for every element of $C$. By the comments above, almost-CL-spaces have the AHP.

\begin{prop}
Every almost-CL-space has the AHP. Besides, the function $\delta$ does not depend on the particular almost-CL-space, $C$ is not the whole dual unit sphere and $\Upsilon$ is the identity.
\end{prop}

We may particularize this result to the main examples of almost-CL-spaces.

\begin{corollary}
All $C(K)$-spaces and all $L_1(\mu)$-spaces have the AHP. Besides, the function $\delta$ does not depend on the particular space, $C$ is not the whole dual unit sphere and $\Upsilon$ is the identity.
\end{corollary}

We do not know if this result extends to general non-separable lush spaces, but a reduction to the separable case of lushness property allows to get this weaker version.

\begin{prop}\label{prop:nonseparable-lush}
There exists a function $\widetilde{\delta}:(0,1)\longrightarrow (0,1)$ such that for every lush space $X$ and every separable subspace $Y$ of $X$, there is a (separable) subspace $Z$ of $X$ containing $Y$ and having the AHP with the function $\widetilde{\delta}$.
\end{prop}

\begin{proof}
Let $\widetilde{\delta}$ the universal function provided in Proposition~\ref{prop-lush-separable}. Let $X$ be a lush space and $Y$ a separable subspace of $X$. We use \cite[Theorem~4.2]{BKMM} to get a separable subspace $Z$ of $X$ containing $Y$ which is lush. Now, Proposition~\ref{prop-lush-separable} gives that $Z$ has the AHP with the function $\widetilde{\delta}$, as required.
\end{proof}

The last family of examples of spaces with the AHSP is given by property $\beta$. A Banach space $X$ has \emph{property $\beta$} if there are two sets $\{ x_i\,:\, i \in I\} \subset S_X$, $\{ x^{*}_i\,:\, i \in I \} \subset S_{X^*}$ and a constant $0 \leq \rho <1$ such that the following conditions hold:
\begin{enumerate}
\item[$(i)$] $x^{*}_i(x_i)=1$, $\forall i \in I $.
\item[$(ii)$] $|x^*_i(x_j)|\leq \rho <1$ if $i, j \in I, i \ne j$.
\item[$(iii)$] $\|x\|=\sup\limits_{i \in I}  \bigl|x^{*}_i(x)\bigr|$ for every $x\in X$.
\end{enumerate}
This property was introduced by J.~Lindenstrauss \cite{Lindens} in his study of norm-attaining operators. We refer to \cite{Moreno} and references therein for more information and background.
It is known that if $X$ has property $\beta$, then $(Z,X)$ has the BPBp for every Banach space $Z$ \cite[Theorem~2.2]{AAGM2}. In particular, $(\ell_1,X)$ has the BPBp and so $X$ has the AHSP. Actually, property $\beta$ implies AHP.

\begin{proposition}\label{prop-betaimpliesAHP}
Property $\beta$ implies property AHP. Besides, the function $\delta$ only depends on the constant $\rho\in [0,1)$ of the definition of property $\beta$, $C$ is not the whole dual unit sphere, and the function $\Upsilon$ is the identity.
\end{proposition}

\begin{proof}
Suppose $X$ has property $\beta$ with constant $\rho\in[0,1)$ and consider the sets $\{ x_i\,:\, i \in I\} \subset
S_X$, $\{ x^{*}_i\,:\, i \in I \} \subset S_{X^*}$ given in the definition of the property. We write $C=\{ x^{*}_i\,:\, i \in I \}$, which is a norming set for $X$, and for $\eps\in (0,1)$, we consider $\delta\in (0,1)$ such that
\begin{equation*}
\frac{\delta(1+\rho)+2\rho\sqrt{2\delta}}{1-(1-\delta)\rho+\rho\sqrt{2\delta}}<\eps.
\end{equation*}
Now, we fix $i\in I$ and consider any $x_0\in S_X$ such that $\re x_i^*(x_0)>1-\delta$. We write
$$
a=\frac{1-\rho}{1-(1-\delta)\rho+\rho\sqrt{2\delta}}\in(0,1]\quad \text{ and } \quad b=x_i^*(x_0).
$$
Then $\re b>1-\delta$ and so $|{\rm Im}\  b| < \sqrt{2\delta}$. Consider the vector
$$
x= a\,x_0 + \bigl(1-a b\bigr)\,x_i\,\in X.
$$
Observe that, clearly, $x^*_i(x)=1$ and that
\begin{align*}
\|x_0-x\| &\leq \left(1-a\right)\|x_0\| + \left|1-a b\right|\|x_i\| \\ & = (1-a) + (1-a \re b)  + a| {\rm Im}\ b|< 1-a + \bigl(1 -a(1-\delta)\bigr)+ a\sqrt{2\delta}\\
&=\frac{\delta(1+\rho)+2\rho\sqrt{2\delta}}{1-(1-\delta)\rho+\rho\sqrt{2\delta}} <\eps.
\end{align*}
It remains to show that $\|x\|=1$ for which it suffices to check that $|x_j^*(x)|\leq 1$ for every $j\neq i$. Indeed, fix $j\in I$, $j\neq i$ and observe that
\begin{align*}
|x_j^*(x)|&\leq a|x_j^*(x_0)| + (1-ab)|x_j^*(x_i)| \leq a + |1-ab|\rho \\
&< a + \bigl(1-a \re b \bigr)\rho +a|{\rm Im}\ b|\rho
\\
&\le  a+ (1-a(1-\delta)) \rho + a\sqrt{2\delta}\rho=1. \qedhere
\end{align*}
\end{proof}

It is now time to present the main result of the section, namely the lifting property of the AHP from $X$ to $L_1(\mu, X)$, which we will use to get a lifting property of the AHSP. Recall that $L_1(\mu, X)$ is the space of all strongly measurable functions $f$ with
\[ \|f\|_1 = \int_\Omega \|f(\omega)\| \, d\mu <\infty.\]
That is, $L_1(\Omega, X)$ is the completion of the space of all simple functions with support of finite measure.

\begin{theorem}
Let $(\Omega,\Sigma,\mu)$ be a measure space and let $X$ be a Banach space. Suppose that there exists a function $\delta$ such that for every separable subspace $Y$ of $X$, there exists a subspace $Z$ of $X$ containing $Y$ and having AHP with function $\delta$. Then, there exists a function $\delta_1$ such that every separable subspace of $L_1(\mu, X)$ is contained in a subspace of $L_1(\mu, X)$ which has the AHP with the function $\delta_1$. Moreover, if $X$ has the AHP, then $L_1(\mu, X)$ has the AHP.
\end{theorem}

\begin{proof}
Let $\tilde{Y}$ be a separable subspace of $L_1(\mu, X)$. Then it is contained in $L_1(\mu, Z)$, where $Z$ is a separable subspace of $X$. So we are done if we assume that the whole space $X$ has the AHP and prove that $L_1(\mu,X)$ does. Let $C$ be the subset of $S_{X^*}$ norming for $X$ and $\delta$ and $\Upsilon_{X,\eps}$ the functions,  given by the definition of the AHP. First, observe that the set $\widetilde{C}$ of those elements in $L_\infty(\mu,X^*)\subset L_1(\mu,X)^*$ of the form
$$
\sum_{i=1}^N x_i^*\chi_{A_i}
$$
where $\{A_i\}_{i=1}^N$ is a disjoint family of measurable subsets with $0<\mu(A_i)<\infty$ and $x_i^*\in C$ for all $i$, is norming for $L_1(\mu,X)$. Now, fix $\eps\in (0,1)$.  For  $\phi=\sum_{i=1}^N x_i^*\chi_{A_i} \in \widetilde{C}$, define
$$
\Upsilon_{L_1(\mu,X),\eps}(\phi)=\sum_{i=1}^N \Upsilon_{X,\eps/2}(x_i^*)\chi_{A_i}
$$
and observe that this definition does not depend on the particular decomposition of $\phi$. Write $\delta_1(\eps)=\frac14 \eps\delta(\eps/2)$. Consider $\phi\in \widetilde{C}$ and a simple function $g\in L_1(\mu,X)$ with $\|g\|=1$ such that $\re \phi(g)>1-\delta_1(\eps)$ (as simple functions are dense in $L_1(\mu,X)$, it is enough tho check the property for them). Then $g$ has the form $g=\sum_{i=1}^N x_i \chi_{A_i}$ where $\{A_i\}_{i=1}^N$ is a disjoint family of measurable subsets with $0<\mu(A_i)<\infty$ and $x_i\in X \setminus \{0\}$ for every $i$. Besides, considering a finer partition if needed, we may and do suppose that $\phi=\sum_{i=1}^N x_i^* \chi_{A_i}$.

Next, let $E = \{ i \,:\, 1\leq i \leq N,\ \re x_i^*(x_i) > (1-\delta(\eps/2)) \|x_i\|\}$. Then, by the AHP of $X$, for each $i\in E$, there is $z_i\in F(\Upsilon_{X,\eps/2}(x_i^*))$ such that
$\left\|z_i - \frac{x_i}{\|x_i\|} \right\|<\varepsilon/2$.
Hence, by setting $y_i = \|x_i\|z_i$, we have
$$
\Upsilon_{X,\eps/2}(x_i^*)(y_i) =  \|x_i\|=\|y_i\| \quad \text{and}\quad \|y_i - x_i \| < \eps/2 \|x_i\|\qquad (i\in E).
$$
By the assumption, we have
\begin{align*}
1-\delta_1(\eps) &< \re \phi(g) = \sum_{i\in E} \re x_i^*(x_i) \mu(A_i) +  \sum_{i\in E^c} \re x_i^*(x_i) \mu(A_i) \\
&\leq  \sum_{i\in E} \re x_i^*(x_i) \mu(A_i)  +  \sum_{i\in E^c} (1-\delta(\eps/2)) \|x_i\|\mu(A_i) \\
&\leq  \sum_{i\in E} \re x_i^*(x_i) \mu(A_i)  + (1-\delta(\eps/2)) \left( 1-\sum_{i\in E} \|x_i\|\mu(A_i)\right)\\
&\leq \delta(\eps/2)  \sum_{i\in E} \re x_i^*(x_i) \mu(A_i)  + 1-\delta(\eps/2),
\end{align*}
where $E^c=\{1,\dots, N\} \setminus E$.
Hence
\[
\sum_{i\in E} \re x_i^*(x_i) \mu(A_i)  > 1- \frac{\delta_1(\eps)}{\delta(\eps/2)}=1-\frac{\eps}4.
\]
In particular, $\beta=  \sum_{i\in E}  \|x_i\| \mu(A_i)>  1-\frac{ \eps}4>0$ and $E$ is not empty. Finally, let $f = \frac{1}{\beta}\sum_{i\in E} y_i \chi_{A_i}$ in $L_1(\mu, X)$. Then $$
\Upsilon_{L_1(\mu,X),\eps}(\phi)(f)=\frac{1}{\beta}\sum_{i\in E} \Upsilon_{X,\eps/2}(x_i^*)(y_i)\mu(A_i)=\frac{1}{\beta}\sum_{i\in E} \|x_i\|\mu(A_i)=1=\|f\|
$$
and
\begin{align*}
\|g - f\|  & \leq \sum_{i\in E} \nor{ \frac{y_i}{\beta} - x_i} \mu(A_i) + \sum_{i\in E^c} \|x_i \| \mu(A_i) \\
&\leq  \sum_{i\in E} \nor{ \frac{y_i}{\beta} - y_i} \mu(A_i) +  \sum_{i\in E} \nor{ x_i - y_i} \mu(A_i) + \sum_{i\in E^c} \|x_i \| \mu(A_i) \\
&\leq  \sum_{i\in E} \frac{\eps}{2} \nor{ x_i} \mu(A_i) + 2(1-\beta)<\eps.\qedhere
\end{align*}
\end{proof}

We particularize the above result to the known examples of spaces with the AHP. This generalizes \cite[Theorem~14]{ChoiKimSK}, where the result was only given for uniformly convex $X$'s.

\begin{corollary} Let $(\Omega,\Sigma,\mu)$ be a measure space and let $X$ be a Banach space. Then $L_1(\mu,X)$ has the AHP (and so the AHSP), provided either of the following holds:
\begin{enumerate}
\item $X$ is finite-dimensional.
\item $X$ is uniformly convex.
\item $X$ is lush and separable.
\item $X$ is an almost-CL-space.
\item $X$ has property $\beta$.
\end{enumerate}
\end{corollary}

For non-separable lush spaces we have the following result.

\begin{corollary}
Let $(\Omega,\Sigma,\mu)$ be a measure space and let $X$ be a (non-separable) lush space. Then every separable subspace of $L_1(\mu, X)$ is contained in a separable subspace of $L_1(\mu, X)$ with the AHP and the function $\delta$ in the definition of the AHP does not depend on subspaces. In particular, $L_1(\mu,X)$ has the AHSP.
\end{corollary}

Let us observe that finite dimensionality, uniform convexity and property $\beta$ do not pass from $X$ to $L_1(\mu,X)$ if $L_1(\mu)$ is non-trivial. Whether $L_1(\mu,X)$ is lush for every lush space $X$ is not known to the best of our knowledge.

\section{On a possible extension of a theorem by Zizler}\label{sec:Zizler}

It is proved in \cite[Example~6.3]{AAGM2} that the classical Lindenstrauss theorem, proving the density of
the set of those operators acting between two arbitrary Banach spaces whose {\em second adjoint} attain the
norm, has no Bishop-Phelps-Bollob\'{a}s counterpart. We may wonder if the result by Zizler providing the density of the set of operators whose {\em first} adjoint attains the norm has a Bishop-Phelps-Bollob\'{a}s counterpart.
More concretely, we may ask about the validity of the
following Bishop-Phelps-Bollob\'{a}s version of Zizler's result:
\begin{quotation}
Given a pair $(X,Y)$ of Banach spaces, is
there a function $\gamma:(0,1)\longrightarrow \R^+$ such that for every $\eps\in (0,1)$, whenever $T_0\in L(X,Y)$ with $\|T_0\|=1$
and $y_0^*\in S_{Y^*}$ satisfy $\|T_0^*(y_0^*)\|>1-\gamma(\eps)$, then there exist $T\in L(X,Y)$
with $\|T\|=1$ and $y^*\in S_{Y^*}$ such that $\|T^*(y^*)\|=1$, $\|y_0^*-y^*\|<\eps$ and $\|T_0-T\|<\eps$?
\end{quotation}

The following easy example shows that this is not always possible.

\begin{example}
{\slshape Let $X$ be a smooth reflexive space whose dual is not super-reflexive and let $Y=\ell_\infty^2$. Then the question
above has a negative answer for $(X,Y)$.} Indeed, since $X$ and $Y$ are reflexive, the above question is equivalent to
whether $(Y^*,X^*)=(\ell_1^2,X^*)$ has the BPBp. Since $X^*$ is strictly convex, this would imply $X^*$ to
be uniformly convex by \cite[Corollary~3.3]{ACKLM-1}. This is impossible since $X^*$ is not super-reflexive.
\end{example}

We observe that it is immediate that if $Y$ is a reflexive space, then for every Banach space $X$ the
question of whether a pair $(X,Y)$ satisfies the BPB version of Zizler result is equivalent to the question of whether
$(Y^*,X^*)$ has the BPBp.  (This follows because every operator from $Y^*$ to $X^*$ is automatically $w^*$-$w^*$-continuous.)

Next, we investigate the Bishop-Phelps-Bollob\'{a}s version of Zizler result when the range space is
$c_0$. We have the following
result, whose proof is based on \cite[Theorem 4.1]{AAGM2}.

\begin{prop}\label{zBP} Given a Banach space $X$, $X^*$ has the AHSP if and only if there is
a function $\gamma:(0,1) \longrightarrow (0,1)$ such that for every $\eps\in (0,1)$, whenever $T_0\in L(X,c_0)$
with $\|T_0\|=1$ and $y_0^*\in S_{c_0^*}$ satisfy $\|T_0^*(y_0^*)\|>1-\gamma(\eps)$, then there exist $T\in L(X,c_0)$
with $\|T\|=1$ and $y^*\in S_{c_0^*}$ such that $\|T^*(y^*)\|=1$, $\|y_0^*-y^*\|<\eps$ and $\|T_0-T\|<\eps$.
\end{prop}

We restate this result as the following corollary.

\begin{corollary}
Let $Y$ be a dual space. Then the pair $(\ell_1,Y)$ has the BPBp if and only if there is a function $\gamma:(0,1) \longrightarrow (0,1)$ such that for every $\eps\in (0,1)$, whenever a $w^*$-$w^*$-continuous $T_0\in S_{L(\ell_1,Y)}$ and
$y_0\in S_{\ell_1}$ satisfy $\|T_0(y_0)\|>1-\gamma(\eps)$, then there exist a $w^*$-$w^*$-continuous $T\in S_{L(\ell_1,Y)}$
and $y\in S_{\ell_1}$ such that $\|T(y)\|=1$, $\|y_0-y\|<\eps$ and $\|T_0-T\|<\eps$ (that is, in
the definition of BPBp we may restrict ourselves to $w^*$-$w^*$-continuous operators.)
\end{corollary}

\begin{proof}[Proof of Proposition~\ref{zBP}]
This proof is based on the one of \cite[Theorem 4.1]{AAGM2}. But for the sake of completeness we give details. For convenience, let $(e_i)$ be the basis of $c_0$ and $(e^*_i)$ be the basis of $\ell_1$.

Suppose that $X^*$ has the AHSP with a function $\eta(\eps)>0$. We fix $T_0\in L(X,c_0)$ with $\|T_0\|=1$ and $y_0^*\in S_{c_0^*}$
satisfying $\|T_0^*(y_0^*)\|>1-\eta(\eps)$. Since the set of finite convex sum of $(e_i^*)$ is dense in $\ell_1$, we may assume
that $y_0^*=\sum_{i=1}^n\alpha_i e^*_i$. Moreover, by composing with an appropriate $w^*$-$w^*$-continuous isometry, we may assume that $\alpha_i\geq 0$ for all $i$.
Since
$$
1-\eta(\eps)<\|T^*_0(y^*_0)\|=\left\|\sum_{i=1}^n \alpha_i T_0^*(e^*_i)\right\|,
$$
we apply the definition of AHSP to
find $A\subset\{1,\ldots ,n\}$, $x^{**}\in S_{X^{**}}$, and $(x_i^*)_{i\in A}$ such
that
$$
\sum_{i\in A}\alpha_i>1-\eps,\qquad \|T^*(e^*_i)-x^*_i\|<\eps \quad \text{and} \quad x^{**}(x^*_i)=1 \quad \bigl(i\in A\bigr).
$$
Define $S\in S_{\mathcal{L}(X,c_0)}$ and $y^*\in S_{\ell_1}$ by
$$
S(x)=\sum_{i\in A}x_i^*(x)e_i+\sum_{i\in \mathbb{N}\backslash A }T^*(e^*_i)(x) e_i \ \ (x\in X), \quad \text{and}\quad y^*=\dfrac{\sum_{i\in A}\alpha_i e^*_i}{\sum_{i\in A}\alpha_i}.
$$
We can see that $S^*(e^*_i)=x^*_i$ for every $i\in A$ and $S^*(e^*_i)=T^*(e^*_i)$ for every $i\in \mathbb{N}\backslash A$.
Therefore, we get that $\|S^*(y^*)\|=1$ and $\|T-S\|<\eps$. Moreover, $\|y^*-y_0^*\|<2\eps$, and so $\gamma(\eps)=\eta(\eps/2)$ completes our proof.\\

For the converse, given $1>\eps>0$, choose $\rho,\eps'\in (0,1)$ such that
$$
0<\sqrt{2(1-\rho)}<\eps/2,\quad  0<\eps'<\eps/2 \quad \text{and} \quad \frac{\eps'}{1-\rho}<\eps/2.
$$
Consider a finite convex series $\sum_{i=1}^{n}\alpha_i$ and a sequence $(x^*_i)_{i=1}^n \subset S_{X^*}$ satisfying
$\|\sum_{i=1}^{n}\alpha_ix^*_i\| > 1-\gamma(\eps')$. Write $z_0^*=\sum_{i=1}^n \alpha_i\,e_i^*\,\in S_{\ell_1}$ and define the operator $S_0\in \mathcal{L}(X,c_0)$ by
$$
S_0(x)=\sum_{i=1}^n x_i^*(x)e_i \quad (x\in X)
$$
which clearly satisfies $\|S_0\|=1$.
Since $\|S^*_0(z_0^*)\|=
\|\sum_{i=1}^{n}\alpha_ix^*_i\|>1-\gamma(\eps')$, by hypothesis there exist $S\in \mathcal{L}(X,c_0)$ with $\|S\|=1$ and $z^*=
\sum_{i=1}^\infty z^*(i)e_i^* \in S_{\ell_1}$ such that
$$
\|S^*(z^*)\|=1,\quad \|S-S_0\|<\eps',\quad \text{and} \quad \|z^*_0-z^*\|<\eps'.
$$
It then follows that
$$
\sum_{i=1}^{n}(\alpha_i- \re z^*(i))<\|z^*_0-z^*\|<\eps',
 $$
and so $\sum_{i=1}^{n} \re
z^*(i)>1-\eps'$. Set $A=\{n\in \{1,\ldots ,n\}~:~\re z^*(i)>\rho |z^*(i)|\}$. Since
\begin{align*}
 1-\eps'
 &<\sum_{i=1}^{n} \re z^*(i)= \sum_{i\in A} \re z^*(i) + \sum_{i\in \{1,\ldots,n\}\backslash A} \re z^*(i)\\
 &\leq \sum_{i\in A} \re z^*(i) + \rho\sum_{i\in \{1,\ldots ,n\}\backslash A} | z^*(i)| \\
 &\leq \sum_{i\in A} \re z^*(i) + \rho\left(1-\sum_{i\in A} | z^*(i)|\right)\\
 &\leq \sum_{i\in A} \re z^*(i) + \rho\left(1-\sum_{i\in A} \re z^*(i)\right),
\end{align*}
we get $\sum_{i\in A} \re z^*(i)>1-\tfrac{\eps' }{1-\rho}$. Hence,
\begin{align*}
\sum_{i\in A} \alpha_i
&\geq \sum_{i\in A} \re z^*(i)-\|z^*_0-z^*\|\\
&\geq 1-\frac{\eps' }{1-\rho}-\eps'>1-\eps.
\end{align*}
On the other hand, choose $x^{**}\in S_{X^{**}}$ so that
$$
1=x^{**}S^*(z^*)=\sum_{i=1}^\infty x^{**}( z^*(i) S^*(e^*_i)).
$$
Set $y^*_i=\tfrac{z^*(i)}{| z^*(i)|}\, S^*(e^*_i) \in S_{X^*}$ for every $i\in A$ and observe that $x^{**}(y_i^*)=1$ for each $i\in A$.
Since
$$
\left|1-\frac{z^*(i)}{| z^*(i)|}\right|<\sqrt{2(1-\rho)}<\eps/2
$$
for every $i\in A$, we get
\begin{align*}
\|x^*_i-y^*_i\|
&=\left\|\frac{z^*(i)}{| z^*(i)|} S^*(e^*_i)-S^*_0(e^*_i)\right\|\\
&\leq \left\|\frac{z^*(i)}{|z^*(i)|} S^*(e^*_i)- S^*(e^*_i)\right\|+\| S^*(e^*_i)-S^*_0(e^*_i)\|\\
&<\eps/2+\eps/2=\eps.
\end{align*}
Finally, $\eta(\eps)=\gamma(\eps')$, the set $A$, the sequence $(y^*_i)_{i\in A}\subset S_{X^*}$ and the functional $x^{**}\in S_{X^{**}}$ complete our proof.
\end{proof}

Very recently, another ``flavor'' of the approximate hyperplane series property has been studied, namely the {\em AHSP for a pair $(X,X^*)$}. This property was introduced in \cite{ABGM} to characterize those Banach spaces $X$ such that $(\ell_1,X)$ has the \emph{BPBp} for bilinear forms.

\begin{definition}[\textrm{\cite{ABGM}}]\label{AHSP2}
Let $X$ be a Banach space. We say that the pair $(X,X^*)$ has the \emph{approximate hyperplane series property for dual pairs}
(\emph{AHSP}) if for every $\eps\in (0,1)$ there exists $0<\eta(\eps)<\eps$ such that for every convex series
$\sum_{n=1}^{\infty}\alpha_k$ and for every sequence $(x^*_k) \subset S_{X^*}$ and $x_0\in S_X$
with
$$
\re \sum_{n=1}^{\infty}\alpha_k x^*_k(x_0) > 1-\eta(\eps)
$$
there exist a subset $A\subset \mathbb{N}$, a subset $\{z_k^*\, :\, k\in A\}\subset S_{X^*}$ and $z_0\in S_X$ satisfying
\begin{itemize}
\item[(1)] $\sum_{k\in A} \alpha_k>1-\eps$,
\item[(2)] $\|z_0-x_0\|<\eps$,
$\|z^*_k-x^*_k\|<\eps$ for all $k\in A$, and $z^*_k(z_0)=1$ for all $k\in A$.
\end{itemize}
\end{definition}

It is clear that by assuming the condition above just for finite sequences $(x_k)$ and $(x^*_k)$, an equivalent
property is obtained. As we already remarked, it is shown in \cite{ABGM} that $(\ell_1, X)$ has the BPBp for
bilinear forms if and only if $(X,X^*)$ has the AHSP. It follows directly from the definition that if a pair $(X,X^*)$ has the AHSP, then $X^*$ has the AHSP. However the converse is not true, since $\ell_\infty=\ell_1^*$ has the AHSP but the pair $(\ell_1,\ell_\infty)$ does not have the AHSP since $(\ell_1,\ell_1)$ does not have the BPBp for bilinear forms \cite{CS}. Moreover, it is known that
no pair of the form $(L_1(\mu),L_1(\mu)^*)$ has the AHSP in the infinite dimensional case \cite{ABGM}. On the other hand, the pair $(X,X^*)$ has the AHSP for the following spaces: finite dimensional $X$, uniformly smooth $X$, $X=C(K)$,
$X=c_0$, and $X=\mathcal{K}(H)$ (the space of compact operators on a Hilbert space $H$) \cite{ABGM}.

The following result characterizes the AHSP for a pair $(X,X^*)$ in terms of the validity of a version of Zizler's result.
Like Proposition~\ref{zBP}, its proof follows the lines of the argument in \cite[Theorem~4.1]{AAGM2}.

\begin{prop}\label{AHH} Given a Banach space $X$, the pair $(X,X^*)$ has the AHSP if
and only if there is a function $\gamma:(0,1) \longrightarrow (0,1)$ such that for every $\eps\in (0,1)$, whenever $T_0\in S_{L(X,c_0)},$ $y_0^*\in S_{c_0^*},$  and $x_0\in S_X$ satisfy
$\re y_0^*T_0(x_0)>1-\gamma(\eps)$, then there exist $T\in S_{L(X,c_0)},$ $y^*\in S_{c_0^*}$
and $x\in S_X$ such that $y^*T(x)=1$, $\|y_0^*-y^*\|<\eps$ $\|x_0-x\|<\eps$ and $\|T_0-T\|<\eps$.
\end{prop}

It is not difficult to adapt the proof of Proposition~\ref{zBP} to this case, so we omit it.

Our last result in this section shows that the pair $(C_0(K,Y), C_0(K,Y)^*)$ has the AHSP when $Y$ is a uniformly smooth space. This generalizes \cite[Corollary~4.5]{ABGM} where the result was proved for $Y=\K$.

\begin{theorem}\label{thm:C(K)AHSP}
Let $K$ be a locally compact Hausdorff space and $Y$ be a uniformly smooth space. Then the pair $(C_0(K,Y), C_0(K,Y)^*)$ has the AHSP.
\end{theorem}

\begin{proof}
We show that for every $\eps>0$, there is $\eta>0$ such that for every $f_0\in S_{C_0(K,Y)}$ there  is $f_1\in S_{C_0(K,Y)}$ satisfying
\begin{enumerate}
\item $\|f_0 - f_1 \|<\eps$.
\item If $\phi \in S_{C_0(K,Y)^*}$ satisfies $\re \phi(f_0) > 1-\eta$, then $\dist(\phi, D(f_1))<5\eps$, where\\ $D(f_1) = \{ \psi\in S_{C_0(K,Y)^*} : \psi(f_1)=1\}$.
\end{enumerate}
Then the result follows from Corollary~3.4 of \cite{ABGM}.

Since $Y$ is uniformly smooth, $Y^*$ has the Radon-Nikod\'{y}m property, and since $C_0(K)^*$ has the Approximation property, $C_0(K, Y)^* =( C_0(K) \hat{\otimes}_\eps Y )^* = C_0(K)^*\hat{\otimes }_\pi Y^*$ \cite[Theorem~5.~33]{Ryan} (where $\hat{\otimes }_\pi$ and $\hat{\otimes}_\eps$ denote the projective and injective tensor product, respectively). Let $r(\eps) = \min \{ \eps/2, 2\delta_{Y^*}(\eps), 1/4\}$ and $\eta(\eps) = \eps r(\eps)^2$ for all $\eps\in (0,1)$, where $\delta_{Y^*}(\eps)$ is the modulus of uniform convexity of $Y^*$.
Fix $\eps\in (0,1)$ and set $\eta = \eta(\eps)$ and $r=r(\eps)$. Suppose that  $\re \phi(f_0)>1 -\eta$ for some $\phi\in S_{C_0(K,Y)^*}$ and for some $f_0\in B_{C_0(K,Y)}$. Then by the definition of the projective tensor product,  there exists  $\phi_1 = \sum_{j=1}^n \alpha_j \mu_j \otimes y^*_j$ such that $\re \phi_1(f_0)>1-\eta$,
$\|\phi_1 - \phi\|<\eps$,  $\sum_{j=1}^n \alpha_j=1$,  $\alpha_j\geq 0$, $\|y_j^*\|=\|\mu_j\|=1$ for all $j=1,\dots, n$.

Let $L= \{ t\in K : \|f_0(t)\|\geq 1-r\}$ and $U=  \{ t\in K : \|f_0(t)\|>1-2r\}$. By the Urysohn lemma, there exists a continuous function $m:K\longrightarrow [0,1]$ such that $m(t)=1$ for all $t\in L$ and $m(t)=0$ for all $t\in K\setminus U$.  Define a function $f_1\in C_0(K,Y)$ by
\begin{align*}
f_1(t)&=\frac{f_0(t)}{\|f_0(t)\|} m(t) + (1-m(t))f_0(t) & \text{if $t\in U$,}\qquad \qquad \qquad \qquad \qquad \qquad \qquad \qquad \\ f_1(t)&=f_0(t) & \text{if $t\notin U$.}\qquad \qquad \qquad \qquad \qquad \qquad \qquad \qquad
\end{align*}
Then $\|f_1(t) \| \leq 1$ for all $t\in K$ and $\|f_1(t)\|=1$ for all $t\in L$. Also,
$$
\|f_0(t) - f_1(t) \| \leq 2r\leq \eps \quad \text{for all $t\in K$.}
$$
Indeed, if $t\in U$, then
$$\|f_0(t) - f_1(t) \| = \nor{  \frac{f_0(t)}{\|f_0(t)\|}m(t) - f_0(t)m(t) } = \bigl| \|f_0(t)\|-1 \bigr| m(t) \leq 2r\leq \eps
$$
and if $t\in K\setminus U$, then $f_0(t)=f_1(t)$.

We assume that $d\mu_j = h_j d|\mu_j|$ for some Borel measurable function $h_j$ with $|h_j|=1$ for all $j=1, \dots, n$.  Then
$$
\re \phi(f_0) = \re \sum_{j=1}^n \alpha_j \int_K y_j^*(f_0(t)) h_j(t)\,  d|\mu_j|(t) > 1-\eta.
$$
Write
\[
A=\left\{ 1\leq j \leq n \,:\, \re \int_K y_j^*(f_0(t)) h_j(t)\,  d|\mu_j|(t)>1-r^2\right\}.
\]
We have $\sum_{j\in A} \alpha_j >1-\eta/r^2 = 1-\eps$. Now, for each $j\in A$, write
$$
B_j = \bigl\{ t\in K \,:\, \re [y^*_j(f_0(t))h_j(t)]>1-r\bigr\}.
$$
Then, for each $j\in A$, we have
\begin{align*}
1-r^2 &< \re \int_K y_j^*(f_0(t)) h_j(t)\,  d|\mu_j|(t)\\
&= \re \int_{B_j}  y_j^*(f_0(t)) h_j(t)\,  d|\mu_j|(t) + \re \int_{K\setminus B_j} y_j^*(f_0(t)) h_j(t)\,  d|\mu_j|(t)\\
&\leq \re \int_{B_j}  y_j^*(f_0(t)) h_j(t)\,  d|\mu_j|(t) +  \int_{K\setminus B_j}(1-r) \,  d|\mu_j|(t)\\
&\leq |\mu_j|(B_j) + (1-r) |\mu_j|(K\setminus B_j) = 1 - r\bigl(1-|\mu_j|(B_j)\bigr)
\end{align*}
and so, $|\mu_j|(B_j) > 1-r$.
By the regularity of the measures, there exists a compact set $K_j\subset B_j$ such that $|\mu_j|(K_j) > 1-r$. Set $\tilde{K} = \bigcup_{j\in A} K_j$ and observe that $\tilde K\subset L$.

For each $t\in U$, there exists a unique $f_0^*(t)\in S_{Y^*}$ such that $\inner{f_0^*(t), f_0(t)}= \|f_0(t)\|$. If $j\in A$ and $t\in B_j$, we have
\[
\re h_j(t) y_j^*(f_0(t)) > 1-r \geq (1-r)\|f_0(t)\|
\]
 and
\[
\re \inner{f_0^*(t), f_0(t)} = \|f_0(t)\|.
\]
 Hence
\[
 \re \frac{  h_j(t) y_j^* + f_0^*(t) }{2} \left(\frac{f_0(t)}{\|f_0(t)\|}\right) \geq 1-\frac r2 \geq 1-\delta_{Y^*}(\eps).
 \]
Then $\|h_j(t)y_j^*-f_0^*(t)\|_{Y^*}\leq \eps$ for all $t\in B_j$ by uniform convexity of $Y^*$.

On the other hand, and again since $Y$ is uniformly smooth, we have for each $t\in U$ and for each $f\in C_0(K,Y)$, that
\[
\inner{f_0^*(t), f(t)} = \lim_{\lambda \to 0} \frac{\left\| \lambda f(t) + \frac{f_0(t)}{\|f_0(t)\|} \right\|-1}{\lambda }.
\]
Hence the function $t\longmapsto \inner{f_0(t)^*, f(t)}$ is Borel measurable on $U$ for each $f\in C_0(K,Y)$.
For each $j\in A$, define
\[ z_j^*(f) = \frac{1}{|\mu_j|(K_j)} \int_{K_j} \inner{f_0^*(t), f(t)} d|\mu_j|(t).\]
Then $z_j^*\in B_{C_0(K,Y)^*}$ for every $j\in A$. Notice that, for each $j\in A$ and $t\in K_j$, one has $t\in L$ and $m(t)=1$. So $f_1(t)=\frac{f_0(t)}{\|f_0(t)\|}$ and $\inner{f_0^*(t), f_1(t)} = 1$. Therefore, $z_j^*(f_1) =1=\|z_j^*\|$ for all $j\in A$.

We claim that $\|z_j^* -\mu_j\otimes y_j^*\| \leq 3\eps$ for every $j\in A$. Indeed,
for $f\in C_0(K,Y)$ with $\|f\|\leq 1$, we have
\begin{align*}
|z_j^*(f) - \mu_j \otimes y_j^*(f) | &= \left|  \frac{1}{|\mu_j|(K_j)} \int_{K_j} \inner{f_0^*(t), f(t)} d|\mu_j|(t) - \int_K y_j^*(f(t))h_j(t) \, d|\mu_j|(t)   \right|\\
&\leq \left| \int_{K\setminus K_j} y_j^*(f(t))h_j(t) \, d|\mu_j|(t)\right| +  \int_{K_j}  | \inner{ f_0^*(t), f(t)} -y_j^*(f(t))h_j(t) | \, d|\mu_j|(t)  \\
&\ \ \ \ + \left (\frac{1}{|\mu_j|(K_j)} -1\right)  \int_{K_j}  |\inner{f_0^*(t), f(t)} |\,\,  d|\mu_j|(t) \\
&  \leq |\mu_j|(K\setminus K_j) +   \int_{K_j}  \| f_0^*(t)  - h_j(t)y_j^* \|_{Y^*} \, d|\mu_j|(t)   + (1- |\mu_j|(K_j))\\
&\leq r+ \eps + r\leq 3\eps.
\end{align*}
Consider $\psi = \frac{1}{\sum_{j\in A} \alpha_j } \sum_{j\in A} \alpha_j z^*_j$ and observe that $\psi\in S_{C_0(K,Y)^*}$ and  $\psi(f_1)=1$. Finally, we get $\|\psi - \phi \|\leq 5\eps$ because
\begin{align*}
\|\psi - \phi_1\| &\leq  \left\| \frac{1}{\sum_{j\in A} \alpha_j } \sum_{j\in A} \alpha_j z^*_j - \sum_{j\in A} \alpha z_j^*  \right\| + \left\| \sum_{j\in A} \alpha_j z_j^* - \sum_{j\in A} \alpha_j \mu_j\otimes y_j^* \right\| + \left\|\sum_{j\in A^c}\alpha_j \mu_j\otimes y_j^*\right\|\\
&\leq  \sum_{j\in A^c} \alpha_j  + 3\eps + \sum_{j\in A^c} \alpha_j \leq 3\eps + 2r\leq 5\eps.\qedhere
\end{align*}
\end{proof}

\end{document}